\documentclass[reqno]{amsart}

\usepackage{myPackages}
\usepackage{myCommands}
\usepackage{zetaCommands}
\usepackage[all]{xy}

\usetikzlibrary{intersections, calc, positioning, arrows.meta, decorations.pathreplacing, decorations.pathmorphing}

\usepackage{tikz-cd}

\newcommand{\cE}{\mathcal{E}}

\newcommand{\bdeta}{\boldsymbol{\eta}}

\newcommand{\bdsigma}{\boldsymbol{\sigma}}

\DeclareMathOperator{\Sym}{Sym}
\DeclareMathOperator{\Tot}{Tot}
\DeclareMathOperator*{\reglim}{reglim}

\newcommand{\locQ}{\underline{\bQ}}

\newcommand{\aUb}{{}_{a\mathstrut}U_{b\mathstrut}}
\newcommand{\aKb}[1]{{}_{a\mathstrut}K_{b\mathstrut}\langle{#1}\rangle}
\newcommand{\aKa}[1]{{}_{a\mathstrut}K_{a\mathstrut}\langle{#1}\rangle}
\newcommand{\aKbp}[1]{{}_{a\mathstrut}K_{b'\mathstrut}\langle{#1}\rangle}
\newcommand{\atKb}[1]{{}_{a\mathstrut}\widetilde{K}_{b\mathstrut}\langle{#1}\rangle}

\newcommand{\aK}[1]{{}_{a\mathstrut}K_{\bullet\mathstrut}\langle{#1}\rangle}
\newcommand{\atK}[1]{{}_{a\mathstrut}\widetilde{K}_{\bullet\mathstrut}\langle{#1}\rangle}

\newcommand{\hgf}[5]{{\mathchoice%
{{}_{#1}F_{#2}\!\left(\genfrac{}{}{0pt}{}{#3}{#4}\,;\,#5\right)}
{{}_{#1}F_{#2}(#3;#4;#5)}
{}{}}}

\begin{document}

\title{
Iterated integrals on the Legendre family of elliptic curves
}
\author{
Eisuke Otsuka
}
\date{}

\address{Eisuke Otsuka\\
Mathematical Institute, Tohoku University.\\
6-3, Aoba, Aramaki, Aoba-ku, Sendai, 980-8578, JAPAN}
\email{eisuke.otsuka.p3@dc.tohoku.ac.jp}

\begin{abstract}

    K.T. Chen showed that iterated integrals give comparison isomorphisms between the cohomologies of bar complexes and fundamental group rings. This led to the development of an algebraic-geometric approach to studying periods given by iterated integrals. In this paper we consider an analogue of this comparison isomorphism theorem for iterated integrals on the Legendre family.
  
\end{abstract}

\maketitle

\tableofcontents

\section{Introduction}
An iterated integral is a multiple integral whose integral region is the $k$-simplex $\Delta_k:=\{(t_1,\dots,t_k)\in(0,1)^k \mid 0<t_1<\cdots<t_k<1\}$. For a piecewise smooth path $\gamma:[0,1]\to M$ and differential forms $\omega_1,\dots,\omega_k$ on a manifold $M$, the iterated integral of $\omega_1,\dots,\omega_k$ along $\gamma$ is defined by

\begin{align}
    \int_\gamma\omega_1\circ\cdots\circ\omega_k:=\int_{\Delta_k}\gamma^\ast\omega_1(t_1)\cdots\gamma^\ast\omega_k(t_k).
\end{align}

It was introduced by K.T. Chen (\cite{Ch73}) to construct the de Rham theory of the loop space of $M$. In the course of his study, he also showed that any iterated integrals can be used to extract non-commutative information for the fundamental group. We will first recall this result. Let $\Omega^\bullet_M$ be the de Rham complex of $M$ and $B^\bullet(\Omega^\bullet_M)$ be the bar complex (see \cref{dfn:bar_complex}). Its cohomology $H^i(B^\bullet(\Omega^\bullet_M))$ has a filtration by length, which is denoted by $L^{-N}H^i(B^\bullet(\Omega^\bullet_M))$. Then, Chen proved the following theorem, called $\pi_1$-de Rham theorem.

\begin{thm}[{\cite[p.92]{Ch75}}] \label{thm:pi1-deRham_Chen_loop}
    For each $N\in\bZge{0}$, the integration of iterated integrals induces a Hopf algebra isomorphism
    \begin{align}
        L^{-N}H^0(B^\bullet(\Omega^\bullet_M))\cong \left(\bQ[\pi_1(M;x)]/J^{N+1}\otimes_\bQ\bC\right)^\vee,
    \end{align}
    where $J\subseteq \bQ[\pi_1(M;x)]$ is the augmentation ideal, which is defined by the kernel of the counit map
    $$\varepsilon:\bQ[\pi_1(M;x)]\to \bQ, \,\sum_ic_i\gamma_i\mapsto\sum_ic_i.$$
    
\end{thm}

We can also extend this theorem to the set of homotopy classes of continuous maps from $[0,1]$ to $M$, denoted by $\pi_1(M;x,y):=\{\gamma:[0,1]\to M \mid \gamma(0)=x, \gamma(1)=y\}/(\text{homotop})$, since there is an isomorphism $\bQ[\pi_1(M;x,y)]\cong\bQ[\pi_1(M;x)]$ given by the concatenation of a fixed path $\gamma_0\in \pi_1(M;x,y)$. 

\begin{thm}[{\cite[p.92]{Ch75}}] \label{thm:pi1-deRham_Chen_path}
    For each $N\in\bZge{0}$, the integration of iterated integrals induces a Hopf algebra isomorphism
    \begin{align}
        L^{-N}H^0(B^\bullet(\Omega^\bullet_M))\cong \left(\bQ[\pi_1(M;x,y)]/J^{N+1}\bQ[\pi_1(M;x,y)]\otimes_\bQ\bC\right)^\vee,
    \end{align}
    where $J\subseteq \bQ[\pi_1(M;x)]$ is the augmentation.
    
\end{thm}

These isomorphisms can be viewed as an iterated (non-abelian) generalization of the classical de Rham theorem
\begin{align}
    H^1_\dR(M,\bC)\cong H^1(M,\bQ)\otimes_\bQ\mathbb{C}.
\end{align}
They give comparison isomorphisms of de Rham and Betti cohomologies in period theory and provide a basis for investigating periods defined by iterated integrals. Also, Morgan \cite{Mo78} and Hain \cite{Ha87} showed that the comparison isomorphism carries a mixed Hodge--Tate structure naturally, which gives a penetrating insight gives a major impact for the study of iterated integral and algebraic geometry (especially motives theory).

In this paper, we will consider an analogue of this theorem to investigate periods of a family of elliptic curves. Let $$\cE:= \left\{([X:Y:Z],\lambda)\in\bP^2\times (\bP^1\setminus\{0,1,\infty\}) \;\middle|\; Y^2Z=X(Z-X)(Z -\lambda X)\right\}$$ be a family of elliptic curves with a fiber bundle $\pi:\cE\to S:=\bP^1\setminus\{0,1,\infty\},\;\pi([X:Y:Z],\lambda)=\lambda$, which is called the Legendre family. Then, for each $\lambda \in S$, the fiber $$E_\lambda:=\left\{[X:Y:Z]\in\bP^2 \;\middle|\; Y^2Z=X(Z-X)(Z-\lambda X)\right\}$$ is an elliptic curve. Then, the Riemann--Hilbert correspondence (see \cite{De}) gives an isomorphism
\begin{align}
    \mathcal{H}^1_\dR(\cE/S)\cong R^1\pi_\ast\locQ_S\otimes_\bQ \mathcal{O}_S.
\end{align}
Here, the left hand side $\mathcal{H}^1_\dR(\cE/S)$ is the de Rham cohomology sheaf on $S$ whose stalk at $\lambda\in S$ is $\mathbb{H}^1_\dR(E_\lambda)$. The right hand side $R^1\pi_\ast\locQ_S$ is the locally constant sheaf whose stalk at $\lambda\in S$ is $H^1(E_\lambda,\bQ)$.

We construct its iteration to consider the following iterated integral. Fix a non-empty open subset $U\subseteq S$ and a point $b\in U$. For $\eta_{i,j}\in \Gamma(U,\mathcal{H}^1(\cE/S))$, $\sigma_{i,j}\in H_1(E_b,\bQ)$, $\omega_i\in H^1_\dR(U)$ ($i=1,\dots,r$, $j=1,\dots,k_i$), and $\gamma:[0,1]\to S$, consider iterated integrals of the forms
\begin{align}
    \left\langle[\bdeta_1;\omega_1|\cdots|\bdeta_r;\omega_r], [\bdsigma_1|\cdots|\bdsigma_r]\otimes\gamma\right\rangle:=\int_\gamma (f_1(\lambda)\omega_1(\lambda))\circ\cdots\circ(f_r(\lambda)\omega_r(\lambda)),
\end{align}
where $f_i(\lambda):=\prod_{j=1}^{k_i}\int_{\widetilde{\sigma}_{i,j}}\eta_{i,j}$ and $\widetilde{\sigma}_{i,j}$ is the lift of the cycle $\sigma_{i,j}$ along $\gamma$. Here, it is abbreviated as $\bdeta_i=(\eta_{i,1},\dots,\eta_{i,k_i})$ and $\bdsigma_i=(\sigma_{i,1},\dots,\sigma_{i,k_i})$.

Then, we have the following theorem, which is a "relative version" of \cref{thm:pi1-deRham_Chen_loop,thm:pi1-deRham_Chen_path}. Let $B^\bullet(\Sym \mathcal{H}^1(\cE/S)\otimes_{\mathcal{O}_S}\Omega^\bullet_S)$ be a the bar complex of dg-algebra $\Sym \mathcal{H}^1(\cE/S)\otimes_{\mathcal{O}_S}\Omega^\bullet_S$ whose differential $d:\Sym \mathcal{H}^1(\cE/S)\otimes_{\mathcal{O}_S}\Omega^r_S\to\Sym \mathcal{H}^1(\cE/S)\otimes_{\mathcal{O}_S}\Omega^{r+1}_S$ is defined by 
\begin{align}
    d[\eta_1\cdots\eta_k;\omega]:=\sum_{j=1}^k[\eta_1\cdots\eta_{j-1}(d_\lambda\eta_j)\eta_{j+1}\cdots\eta_k;d\lambda\wedge\omega]+[\eta_1\cdots\eta_k;d\omega],
\end{align}
where $d_\lambda$ is the differential with respect to the parameter $\lambda\in S$ determined by the Gauss--Manin connection (see \cite{Ka70}). Then, the 0-th cohomology of the bar complex can be written by $\Sym H_1(E_b,\bQ)$ and $\bQ[\pi_1(S;a,b)]$ for any $a,b\in S$ as follows.

\begin{thm} \label{thm:pi1-deRham_family}
    For each open set $U\subseteq S$, $M, N\in\bZge{0}$ and $a,b\in U$, the integration of iterated integrals induces an isomorphism
    \begin{align} \label{eq:pi1-deRham_family}
        &L^{-N}\mathbb{H}^0(U, B^\bullet(\Sym^{\le M} \mathcal{H}^1(\cE|_U/U)\otimes_{\mathcal{O}_U}\Omega^\bullet_U))\\
        &\cong
        \left((\Sym^{\le M} H_1(E_b,\bQ))^{\otimes \le N}\otimes_\bQ\bQ[\pi_1(U;a,b)]/J^{N+1}\bQ[\pi_1(U;a,b)]\otimes_\bQ\bC\right)^\vee,
    \end{align}
    where $J\subseteq \mathbb{Q}[\pi_1(U;a)]$ is the augmentation ideal. In particular, if $a=b$, we have
    \begin{align}
        &L^{-N}\mathbb{H}^0(U, B^\bullet(\Sym^{\le M} \mathcal{H}^1(\cE|_U/U)\otimes_{\mathcal{O}_U}\Omega^\bullet_U))\\
        &\cong
        \left((\Sym^{\le M} H_1(E_a,\bQ))^{\otimes \le N}\otimes_\bQ\bQ[\pi_1(U;a)]/J^{N+1}\otimes_\bQ\bC\right)^\vee.
    \end{align}
\end{thm}

The left hand side 
$$L^{-N}\mathbb{H}^0(U, B^\bullet(\Sym^{\le M} \mathcal{H}^1(\cE|_U/U)\otimes_{\mathcal{O}_S|_U}\Omega^\bullet_S|_U))$$
is the cohomology of the bar complex $$B^\bullet(\Sym^{\le M} \mathcal{H}^1(\cE|_U/U)\otimes_{\mathcal{O}_S|_U}\Omega^\bullet_S|_U)$$ of the length less than or equal to $N$ part, which is explained in the first part of \cref{subsec:preparation}. If we take $U=S$, we have the following corollary.

\begin{cor} \label{cor:pi1-deRham_family}
    For each $M, N\in\bZge{0}$ and $a,b\in S$, the integration of iterated integrals induces an isomorphism
    \begin{align}
        &L^{-N}\mathbb{H}^0(S, B^\bullet(\Sym^{\le M} \mathcal{H}^1(\cE/S)\otimes_{\mathcal{O}_S}\Omega^\bullet_S))\\
        &\cong
        \left((\Sym^{\le M} H_1(E_b,\bQ))^{\otimes \le N}\otimes_\bQ\bQ[\pi_1(S;a,b)]/J^{N+1}\bQ[\pi_1(S;a,b)]\otimes_\bQ\bC\right)^\vee,
    \end{align}
    where $J\subseteq \mathbb{Q}[\pi_1(S;a)]$ is the augmentation ideal. In particular, if $a=b$, we have
    \begin{align}
        &L^{-N}\mathbb{H}^0(S, B^\bullet(\Sym^{\le M} \mathcal{H}^1(\cE/S)\otimes_{\mathcal{O}_S}\Omega^\bullet_S))\\
        &\cong
        \left((\Sym^{\le M} H_1(E_a,\bQ))^{\otimes \le N}\otimes_\bQ\bQ[\pi_1(S;a)]/J^{N+1}\otimes_\bQ\bC\right)^\vee.
    \end{align}
\end{cor}

In a sense, \cref{thm:pi1-deRham_family} provides an isomorphism between de Rham cohomology (on the left) and Betti cohomology (on the right), and the existence of a motivic fundamental group that induces this isomorphism is expected. For example, the mixed elliptic motives, which is considered in particular detail in \cite{Pa10}, or the universal mixed elliptic motives, which is introduced in \cite{HaMa15}, are considered to be a candidates.

We can also extend \cref{thm:pi1-deRham_family} to the case of $a,b\in\{0,1,\infty\}$ by using the theory of tangential base point and regularized iterated integrals.
\begin{thm} \label{thm:pi1-deRham_family_reg}
    For each open set $U\subseteq S$, $M, N\in\bZge{0}$ and two tangential base points $\mathbf{a}, \mathbf{b}$ on $U$, the integration of regularized iterated integrals induces an isomorphism
    \begin{align} \label{eq:pi1-deRham_family_reg}
        &L^{-N}\mathbb{H}^0(U, B^\bullet(\Sym^{\le M} \mathcal{H}^1(\cE|_U/U)\otimes_{\mathcal{O}_U}\Omega^\bullet_U))\\
        &\cong
        \left((\Sym^{\le M} H_1(E_\mathbf{b},\bQ))^{\otimes \le N}\otimes_\bQ\bQ[\pi_1(U;\mathbf{a},\mathbf{b})]/J^{N+1}\bQ[\pi_1(U;\mathbf{a},\mathbf{b})]\otimes_\bQ\bC\right)^\vee,
    \end{align}
    where $J\subseteq \mathbb{Q}[\pi_1(U;\mathbf{a})]$ is the augmentation ideal.
\end{thm}
This allows us to capture this iterative integral in a form that includes multiple zeta values. It is mentioned in detail in \cref{subsec:general_setting,subsec:regularized_iterated_integral}.

We now organize this paper as follows. In \cref{sec:iterated_integrals}, we will introduce iterated integrals in general setting and describe some properties. After that, we will introduce our target. In \cref{sec:main_theorem}, we will give a proof of \cref{thm:pi1-deRham_family}. We will also describe the regularized version in \cref{subsec:regularized_iterated_integral}. In \cref{app:cosimplicial_obj}, we will introduce cosimplicial object, which is used in the proof of \cref{thm:pi1-deRham_family}.

\section*{Acknowledgement}
The author would like to express his deepest gratitude to his supervisor, Professor Takuya Yamauchi for his insightful comments. The author would like to thank the WISE Program for AI Electronics and AGS RISE Program of Tohoku University for their financial support to carry out our study.

\section{Iterated integrals} \label{sec:iterated_integrals}

In this section, we give a setting of this paper.

\begin{notn}
    The following notation will be used.
    \begin{itemize}
        \item For a finite set $X$, we denote by $\# X$ the cardinality of $X$.
        \item For $N\in \bZge{0}$, we denote $I_N:=\{0,\dots,N\}$.
        \item For $k\in \bZge{1}$, we denote $\Delta_k:=\{(t_1,\dots,t_k)\in[0,1]^k\,|\,0<t_1<\cdots<t_k<1\}$.
        \item For $n\in \bZge{1}$, we denote by $S_n$ the symmetry group of degree $n$.
        \item For two paths $\gamma_0,\gamma_1:[0,1]\to M$ with $\gamma_0(1)=\gamma_1(0)$, the connected path is denoted by $\gamma_0\gamma_1:[0,1]\to M$, which is defined by $$\gamma_0\gamma_1(t):=\begin{cases}
            \gamma_0(2t), & 0<t<\frac{1}{2},\\
            \gamma_1(2t-1), & \frac{1}{2}<t<1.
        \end{cases}$$
        \item For a path $\gamma:[0,1]\to M$, the inverse path is denoted by $\gamma^{-1}:[0,1]\to M$, which is defined by $$\gamma^{-1}(t):=\gamma(1-t).$$
    \end{itemize}
\end{notn}

\subsection{General setting} \label{subsec:general_setting}

First, we will mention some properties of the iterated integral on a general manifold $M$.

\begin{dfn}
    For a piecewise smooth path $\gamma:[0,1]\to M$ and differential forms $\omega_1,\dots,\omega_k$ on a manifold $M$, the iterated integral is defined by
    \begin{align}
        \int_\gamma\omega_1\circ\cdots\circ\omega_k:=\int_{\Delta_k}\gamma^\ast\omega_1(t_1)\cdots\gamma^\ast\omega_k(t_k).
    \end{align}
\end{dfn}

Iterated integrals have the following properties.

\begin{prop} \label{prop:properties_of_iterated_integral}
    The following properties hold.
    \begin{enumerate}
        \item For any path $\gamma:[0,1]\to M$ and closed $1$-forms $\omega_1,\dots,\omega_r$ with $\omega_i\wedge\omega_{i+1}=0$ ($i=1,\dots,r-1$), then $\int_\gamma \omega_1\circ\cdots\circ\omega_r$ is homotopy equivalent.
        
        \item For any path $\gamma:[0,1]\to M$ and $1$-forms $\omega_1,\dots,\omega_r,\omega_{r+1},\dots,\omega_{r+s}$, it holds
        \begin{align}
            \int_\gamma \omega_1\circ\cdots\circ\omega_r\int_\gamma \omega_{r+1}\circ\cdots\circ\omega_{r+s}=\sum_{\sigma\in S_{r,s}} \int_\gamma\omega_{\sigma^{-1}(1)}\circ\cdots\circ\omega_{\sigma^{-1}(r+s)}, 
        \end{align}
        where $S_{r,s}$ is the $(r,s)$-shuffle
        \begin{align}
            S_{r,s}:=\left\{\sigma\in \mathcal{S}_{r+s}\,\middle|\,\begin{array}{c}\sigma(1)<\cdots<\sigma(r)\\\sigma(r+1)<\cdots<\sigma(r+s)\end{array}\right\}.
        \end{align}
        
        \item For any paths $\gamma_0,\gamma_1:[0,1]\to M$ with $\gamma_0(1)\gamma_1(0)$ and $1$-forms $\omega_1,\dots,\omega_r$, it holds
        \begin{align}
            \int_{\gamma_0\gamma_1} \omega_1\circ\cdots\circ\omega_r=\sum_{i=0}^r \int_{\gamma_0}\omega_1\circ\cdots\circ\omega_i\int_{\gamma_1}\omega_{i+1}\circ\cdots\circ\omega_r.
        \end{align} \label{item:path_connection}

        \item For any path $\gamma:[0,1]\to M$ and $1$-forms $\omega_1,\dots,\omega_r$, it holds
        \begin{align}
            \int_{\gamma^{-1}} \omega_1\circ\cdots\circ\omega_r=(-1)^r \int_\gamma\omega_r\circ\cdots\circ\omega_1.
        \end{align}
    \end{enumerate}
\end{prop}

Next, we define tangential base point to define a regularization of iterated integrals.

\begin{dfn}
    A pair $\mathbf{x}=(x,u)$ of a point $x\in M$ and a tangent vector $u\in T_xM\setminus\{0\}$ is called a tangential base point. A (cuspidal) path from a tangential base point $\mathbf{x}=(x,u)$ to a tangential base point $\mathbf{y}=(y,v)$ is a piecewise smooth path $\gamma:[0,1]\to M$ from $x$ to $y$ satisfying
    \begin{align}
        \lim_{t\searrow0}\frac{d\gamma(t)}{dt}=u, &&
        \lim_{t\nearrow1}\frac{d\gamma(t)}{dt}=-v.
    \end{align}
\end{dfn}

In general, integrals can diverge if the differential form is allowed to be logarithmic. In such cases, the integral is defined by fixing the tangent vector of the endpoints of the path and taking the regularized limit of the integral.

\begin{dfn}
    If a function $f:(0,T)\to \bC$ ($T>0$), can be written as
    \begin{align}
        f(t)=f_0(t) + O(\log^k(t))
    \end{align}
    for some $f_0(t)$ with $|f_0(t)|=O(t^\delta)$ ($\delta>0$) as $t\to0$, we say the function $f$ admits a logarithmic asymptotic development. At that time, we define the regularized limit
    \begin{align}
        \reglim_{t\to0}f(t):=\lim_{t\to0}f_0(t).
    \end{align}
\end{dfn}

\begin{dfnprop}[{\cite[Lemma 3.363, p.286]{GiFr}}]
    For a cuspidal path $\gamma:[0,1]\to M$ from a tangential base point $\mathbf{x}=(x,u)$ to a tangential base point $\mathbf{y}=(y,v)$, we define $\gamma_\epsilon:[0,1]\to M$ by $\gamma_\epsilon(t):=\gamma(t(1-\epsilon)+(1-t)\epsilon)$ for sufficiently small $\epsilon>0$. Then, for $1$-forms $\omega_1,\dots,\omega_r$ (allowed to be logarithmic), the function
    \begin{align}
        \epsilon \mapsto \int_{\gamma_\epsilon} \omega_1\circ\cdots\circ\omega_r
    \end{align}
    admits logarithmic asymptotic development. Then, we define the regularized iterated integral as
    \begin{align}
        \int_\gamma\omega_1\circ\cdots\circ\omega_r:=\reglim_{\epsilon\to 0}\int_{\gamma_\epsilon} \omega_1\circ\cdots\circ\omega_r.
    \end{align}
\end{dfnprop}

\begin{ex}
    For $k_1,\dots,k_d\in\bZge{1}$ with $k_d\ge2$, the multiple zeta value is defined by
    \begin{align}
        \zeta(k_1,\dots,k_d):=\sum_{0<n_1<\dots<n_d}\frac{1}{n_1^{k_1}\cdots n_d^{k_d}}.
    \end{align}
    Any multiple zeta values have an iterated integral representation
    \begin{align}\label{eq:MZV_iterated_integral}
        \zeta(k_1,\dots,k_d)=\int_\dch\chi_1\circ\chi_0^{\circ k_1-1}\circ \cdots\circ \chi_1 \circ \chi_0^{\circ k_d-1}
    \end{align}
    on $M=\mathbb{P}^1(\mathbb{C})$ with a cuspidal path $\dch:[0,1]\to M$, $\dch(t)=[t:1]$ and logarithmic $1$-forms $\chi_0:=\frac{dz}{z}$ and $\chi_1:=\frac{dz}{1-z}$. The iterated integral (\ref{eq:MZV_iterated_integral}) converges when the most left form is $\chi_1$ and the most right form is $\chi_0$, and in that case, it became the multiple zeta value by expanding $\chi_1$ as geometric series and calculating the integral term by term. Also, even if the iterated integral diverges, the regularization can give a complex number. It is known that the value of regularized iterated integral can always be written by multiple zeta values in that case.
\end{ex}

The regularized iterated integral has the same properties as \cref{prop:properties_of_iterated_integral}.

\begin{prop} \label{prop:properties_of_regularized_iterated_integral}
    The following properties hold.
    \begin{enumerate}
        \item For any cuspidal path $\gamma:[0,1]\to M$ and logarithmic closed $1$-forms $\omega_1,\dots,\omega_r$ with $\omega_i\wedge\omega_{i+1}=0$ ($i=1,\dots,r-1$), then $\int_\gamma \omega_1\circ\cdots\circ\omega_r$ is homotopy equivalent on the space excluding singularities of differential forms from $M$.
        
        \item \label{item:shuffle_for_reg_iterated_integral} For any cuspidal path $\gamma:[0,1]\to M$ and logarithmic $1$-forms $\omega_1,\dots,\omega_r,\omega_{r+1},\dots,\omega_{r+s}$, it holds
        \begin{align}
            \int_\gamma \omega_1\circ\cdots\circ\omega_r\int_\gamma \omega_{r+1}\circ\cdots\circ\omega_{r+s}=\sum_{\sigma\in S_{r,s}} \int_\gamma\omega_{\sigma^{-1}(1)}\circ\cdots\circ\omega_{\sigma^{-1}(r+s)}. 
        \end{align}

        \item \label{item:connection_for_reg_iterated_integral} For any cuspidal paths $\gamma_0,\gamma_1:[0,1]\to M$ such that the end point of $\gamma_0$ and the start point of $\gamma_1$ coincide (as tangential base points), and $1$-forms $\omega_1,\dots,\omega_r$, it holds
        \begin{align}
            \int_{\gamma_0\gamma_1} \omega_1\circ\cdots\circ\omega_r=\sum_{i=0}^r \int_{\gamma_0}\omega_1\circ\cdots\circ\omega_i\int_{\gamma_1}\omega_{i+1}\circ\cdots\circ\omega_r.
        \end{align}

        \item \label{item:inverse_for_reg_iterated_integral} For any cuspidal path $\gamma:[0,1]\to M$ and $1$-forms $\omega_1,\dots,\omega_r$, it holds
        \begin{align}
            \int_{\gamma^{-1}} \omega_1\circ\cdots\circ\omega_r=(-1)^r \int_\gamma\omega_r\circ\cdots\circ\omega_1.
        \end{align}
    \end{enumerate}    
\end{prop}

\subsection{Iterated integrals on Legendre family}

In this section, we will introduce iterated integrals on Legendre family, which is the main subject of this paper.

Let $\pi:\cE\to S$ be the Legendre family defined by $$\cE:= \left\{([X:Y:Z],\lambda)\in\bP^2\times (\bP^1\setminus\{0,1,\infty\}) \;\middle|\; Y^2Z=X(Z-X)(Z-\lambda X)\right\}$$ with a fiber bundle $\pi:\cE\to S:=\bP^1\setminus\{0,1,\infty\},\;\pi([X:Y:Z],\lambda)=\lambda$. We consider the following iterated integral. For $\eta_{i,j}\in \Gamma(S,\mathcal{H}^1(\cE/S))$, $\sigma_{i,j}\in H_1(E_b,\bQ)$, $\omega_i\in H^1_\dR(S)$ ($i=1,\dots,r$, $j=1,\dots,k_i$), and $\gamma:[0,1]\to S$, consider iterated integrals of the forms
\begin{align}
    \left\langle[\bdeta_1;\omega_1|\cdots|\bdeta_r;\omega_r], [\bdsigma_1|\cdots|\bdsigma_r]\otimes\gamma\right\rangle:=\int_\gamma (f_1(\lambda)\omega_1(\lambda))\circ\cdots\circ(f_r(\lambda)\omega_r(\lambda)),
\end{align}
where $f_i(\lambda):=\prod_{j=1}^{k_i}\int_{\widetilde{\sigma}_{i,j}}\eta_{i,j}$ and $\widetilde{\sigma}_{i,j}$ is the lift of the cycle $\sigma_{i,j}$ along $\gamma$. Here, it is abbreviated as $\bdeta_i=(\eta_{i,1},\dots,\eta_{i,k_i})$ and $\bdsigma_i=(\sigma_{i,1},\dots,\sigma_{i,k_i})$.

\begin{rem}
    \begin{enumerate}
        \item In the following, we fix the basis of $\Gamma(S,\mathcal{H}^1(\cE/S))$ as $\left\{\theta_0:=\frac{dx}{y}, \theta_1:=\frac{xdx}{y}\right\}$. Here, we show in the affine coordinates $x=\frac{X}{Z}$ and $y=\frac{Y}{Z}$ through the isomorphism
        \begin{align}
            \Gamma(U,\mathcal{H}^1(\cE/S))\cong \Gamma(U,\mathcal{H}^1(\cE/S \setminus \{O_{\cE/S}\})),
        \end{align}
        where $O_{\cE/S}$ is the zero section of $\pi$. Also, the basis $\{\alpha,\beta\}$ of $H_1(E_\lambda,\bQ)$ is taken as follows (see FIGURE \ref{fig:cycles}).
        \begin{itemize}
            \item $\alpha$ is a cycle such that the interval $[0,1]$ is followed on one sheet and then the reverse path is followed back on the other sheet.
            \item $\beta$ is a cycle such that the interval $[0,\lambda^{-1}]$ is followed on one sheet and then the reverse path is followed back on the other sheet.
        \end{itemize}
            Then, if $|\lambda|<1$, it holds
        \begin{align}
            f_0(\lambda):=\int_{\widetilde{\alpha}}\theta_0&=\widetilde{K}(k), & g_0(\lambda):=\int_{\widetilde{\beta}}\theta_0&=\sqrt{-1}\widetilde{K}(k'),\\
            f_1(\lambda):=\int_{\widetilde{\alpha}}\theta_1&=\widetilde{E}(k), & g_1(\lambda):=\int_{\widetilde{\beta}}\theta_1&=\sqrt{-1}\widetilde{E}(k'),
        \end{align}
            where $k=\sqrt{\lambda}$, $k'=\sqrt{1-\lambda}$. In addition, $\widetilde{K}$ and $\widetilde{E}$ are defined by the Gauss hypergeometric function
        \begin{align}
            \hgf{2}{1}{\alpha,\beta}{\gamma}{z}:=\sum_{n\ge0}\frac{(\alpha)_n(\beta)_n}{(\gamma)_nn!}z^n
        \end{align}
        as
        \begin{align}
            \widetilde{K}(k):=\pi\cdot\hgf{2}{1}{\frac{1}{2},\frac{1}{2}}{1}{\lambda}, && \widetilde{E}(k):=\pi\cdot\hgf{2}{1}{\frac{1}{2},\frac{3}{2}}{2}{\lambda},
        \end{align}
        where $(\alpha)_n:=\alpha(\alpha+1)\cdots(\alpha+n-1)$ is the Pochhammer symbol.
        \item We fix the basis of $H^1_\dR(S)$ as $\left\{\chi_0:=\frac{d\lambda}{\lambda}, \chi_1:=\frac{d\lambda}{1-\lambda}\right\}$.
    \end{enumerate}
\end{rem}

\begin{figure}[htbp]
    \centering
    \begin{tikzpicture}[scale=2, line join=round, >=latex, every node/.style={font=\footnotesize}]
    \draw[thick] (0,0) ellipse (2 and 1);

    \draw[thick]  (0.9,0.05) arc (30:153:1 and 0.5);
    \draw[thick] (-1,0.25) arc (180:360:1 and 0.5);

    \node (b) at (-1.1,-0.4) {};
    \node[left] at (b) {$x=0$};
    \fill (b) circle (1.1pt);

    \node (1) at (1.1,0.4) {};
    \node[right] at (1) {$x=\lambda^{-1}$};
    \fill (1) circle (1.1pt);

    \node (0) at (-0.85,-0.7) {};
    \node[below right] at (0) {$x=1$};
    \fill (0) circle (1.1pt);

    \draw[thick, dashed] (0,0) ellipse (1.33 and 0.7);
    \node at (1.4,-0.2) {$\beta$};

    \draw[thick, dashed, rotate=45] (-1.4,0.3) arc (180:0:0.45 and 0.2);
    \draw[thick, dotted, rotate=225] (0.5,-0.3) arc (180:0:0.45 and 0.2);
    \node at (-0.95,-0.15) {$\alpha$};




    \end{tikzpicture}    

    \caption{Cycles of $E_b$}  \label{fig:cycles}
\end{figure}

\begin{ex}
    \begin{enumerate}
        \item If $k_i=0$ for any $i$ and $\gamma=\dch$, we have
        \begin{align}
            \langle[\;;\omega_1|\cdots|\;;\omega_r],[\underbrace{\;|\cdots|\;}_{r}]\otimes\dch\rangle=\int_\dch\omega_1\circ\cdots\circ\omega_r
        \end{align}
        for $\omega_1,\dots\omega_r\in\left\{\chi_0,\chi_1\right\}$, which is the iterated integral representation of multiple zeta values.
        
        \item For a positive integer $r\ge2$, the value
        \begin{align}
            \pi\cdot\sum_{0\le n_1<n_2}\frac{\left(\frac{1}{2}\right)_{n_1}^2}{(n_1!)^2n_2^{r-1}},
        \end{align}
        can be written as
        \begin{align}
            \left\langle[\theta_0;\chi_1|\underbrace{\;;\chi_0|\cdots|\;;\chi_0}_{r-1}],[\alpha|\underbrace{\;|\cdots|\;}_{r-1}]\otimes\dch\right\rangle
        \end{align}
        since the Taylor expansion of $f_0$ is
        \begin{align}
            f_0(\lambda)=\pi\cdot \sum_{n\ge0}\frac{\left(\frac{1}{2}\right)_{n}^2}{(n!)^2}
        \end{align}
        and calculating the integral term by term. Similarly, the value
        \begin{align}
            \pi\cdot\sum_{0\le n_1<n_2}\frac{\left(\frac{1}{2}\right)_{n_1}\left(\frac{3}{2}\right)_{n_1}}{n_1!(n_1+1)!n_2^{r-1}},
        \end{align}
        can be written as
        \begin{align}
            \left\langle[\theta_1;\chi_1|\underbrace{\;;\chi_0|\cdots|\;;\chi_0}_{r-1}],[\alpha|\underbrace{\;|\cdots|\;}_{r-1}]\otimes\dch\right\rangle.
        \end{align}
        Then, these can be treated as periods on the Legendre family in our setting.
    \end{enumerate}
\end{ex}

From \cref{prop:properties_of_regularized_iterated_integral}, our periods have some relations as follows.

\begin{prop}
    Let $\eta_{i,j}\in \Gamma(U,\mathcal{H}^1(\cE/S))$, $\sigma_{i,j}\in H_1(E_b,\bQ)$ and $\omega_i\in H^1_\dR(U)$.
    \begin{enumerate}
        \item \label{item:shuffle_rel} It holds
        \begin{align}
            &\left\langle[\bdeta_1;\omega_1|\cdots|\bdeta_r;\omega_r], [\bdsigma_1|\cdots|\bdsigma_r]\otimes\dch\right\rangle\left\langle[\bdeta_{r+1};\omega_{r+1}|\cdots|\bdeta_{r+s};\omega_{r+s}], [\bdsigma_{r+1}|\cdots|\bdsigma_{r+s}]\otimes\dch\right\rangle\\
            &=\sum_{\sigma\in S_{r,s}}\left\langle[\bdeta_{\sigma^{-1}(1)};\omega_{\sigma^{-1}(1)}|\cdots|\bdeta_{\sigma^{-1}(r+s)};\omega_{\sigma^{-1}(r+s)}], [\bdsigma_{\sigma^{-1}(1)}|\cdots|\bdsigma_{\sigma^{-1}(r+s)}]\otimes\dch\right\rangle.
        \end{align}
        \item \label{item:dual_rel} It holds
        \begin{align}
            \left\langle[\bdeta_1;\omega_1|\cdots|\bdeta_r;\omega_r], [\bdsigma_1|\cdots|\bdsigma_r]\otimes\dch\right\rangle=\left\langle[\bdeta_r;\omega_r|\cdots|\bdeta_1;\omega_1], [\bdsigma_r'|\cdots|\bdsigma_1']\otimes\dch\right\rangle,
        \end{align}
        where $\bdeta'_i=(\eta_{i,1}',\dots,\eta_{i,k_i}')$ is determined by
        \begin{align}
            \alpha'=-\sqrt{-1}\beta, && \beta'=\sqrt{-1}\alpha.
        \end{align}
    \end{enumerate}
\end{prop}
\begin{proof}
    (\ref{item:shuffle_rel}) is proved directly from (\ref{item:shuffle_for_reg_iterated_integral}) of \cref{prop:properties_of_regularized_iterated_integral}. (\ref{item:dual_rel}) is proved from (\ref{item:inverse_for_reg_iterated_integral}) of \cref{prop:properties_of_regularized_iterated_integral} and
    \begin{align}
        g_i(1-\lambda)=\sqrt{-1}f_i(\lambda)
    \end{align} 
    for $i=0,1$.
\end{proof}

\section{An analogue of \texorpdfstring{$\pi_1$}-de Rham theorem} \label{sec:main_theorem}

In this section, we give a strict statement and a proof of \cref{thm:pi1-deRham_family}.

\subsection{Preparation} \label{subsec:preparation}
First, we define the bar complex. Let $k$ be an algebraic field.

\begin{dfn} \label{dfn:bar_complex}
    For a dg-algebra $A=(A^\bullet,d)$ on $k$, we define the bar complex $B^\bullet(A)$ by
    \begin{align}
        B^p(A):=\bigoplus_{r\ge0}\bigoplus_{p_1+\cdots+p_r=p-r} A^{p_1}\otimes\cdots \otimes A^{p_r}.
    \end{align}
    The differential $d:B^p(A)\to B^{p+1}(A)$ is defined by
    \begin{align} \label{eq:diff_of_bar_cpx}
      \begin{split}
        d[a_1|\cdots|a_r]&:=\sum_{i=1}^{r-1}(-1)^{\nu_{i-1}+1}[a_1|\cdots|a_{i-1}|d a_i|a_{i+1}|\cdots |a_r]\\
        &+\sum_{i=1}^{r}(-1)^{\nu_{i}+1}[a_1|\cdots|a_{i-1}|a_i\wedge a_{i+1}|a_{i+2}|\cdots |a_r],
      \end{split}
    \end{align}
    where $\nu_i=p_1+\cdots+p_i-i$ for $a_i\in A^{p_i}$.
\end{dfn}

\begin{rem}
    The differential given by the \cref{eq:diff_of_bar_cpx} comes from the differential structure of iterated integral of general differential forms.
\end{rem}

The bar complex $B^\bullet(A)$ has the following extra structures.

\begin{enumerate}
    \item (Length filtration) We have an decreasing filtration $$L^{N}B^\bullet(A)\subseteq B^\bullet(A), $$ which is generated by elements $[a_1|\cdots|a_m]$ with $m\le -N$, which is commutative with the differential.
    \item (Product) We have a commutative product $\shu$ called the shuffle product defined by $$[a_1|\cdots|a_r]\shu[a_{r+1}|\cdots|a_{r+s}]:=\sum_{\sigma\in S_{r,s}} \epsilon(\sigma)[a_{\sigma^{-1}(1)}|\cdots|a_{\sigma^{-1}(r+s)}],$$ where $$S_{r,s}:=\left\{\sigma\in \mathcal{S}_{r+s}\,\middle|\,\begin{array}{c}\sigma(1)<\cdots<\sigma(r)\\\sigma(r+1)<\cdots<\sigma(r+s)\end{array}\right\}$$ and $\epsilon(\sigma)$ is the sign determined by $$a_1\wedge\cdots\wedge a_{r+s}=\epsilon(\sigma)a_{\sigma^{-1}(1)}\wedge\cdots\wedge a_{\sigma^{-1}(r+s)}.$$
    \item (Coproduct) We have a coproduct $\Delta$ defined by $$\Delta([a_1|\cdots|a_r]):=\sum_{i=0}^r[a_1|\cdots|a_i]\otimes[a_{i+1}|\cdots|a_r].$$
    \item (Antipode) We have an antipode defined by $$S([a_1|\cdots|a_r]):=(-1)^n\epsilon(\tau_r)[a_r|\cdots|a_1], $$ where $\tau_r$ is a permutation $\tau_r(i):=r-i$.
\end{enumerate}

\begin{lem}[{\cite[Lemma 3.261, p.249]{GiFr}}] \label{lem:Hopf_strucure_of_bar_complex}
    Let $A$ be a connected commutative dg-algebra, which means $A^0=k$ and $A^n=0$ if $n<0$. Then, the above operation induces a structure on $H^0(B^\bullet(A))$ as a Hopf algebra.
\end{lem}

Now we extend the definition of the bar complex to sheaves.

\begin{dfn}
    For a complex of sheaves $\mathcal{F}^\bullet$ on a topological space $X$, the symbol $B^\bullet(\mathcal{F}^\bullet)$ is used to denote a bar complex of sheaves $U\rightsquigarrow B^\bullet(\mathcal{F}^\bullet(U))$.
\end{dfn}

In our setting, we consider the cohomology of the bar complex of sheaves for $\Sym^{\le M} \mathcal{H}^1(\cE|_U/U)\otimes_{\mathcal{O}_U}\Omega^\bullet_U$. Then, from \cref{lem:Hopf_strucure_of_bar_complex}, $H^0(B^\bullet(\Sym^{\le M} \mathcal{H}^1(\cE|_U/U)\otimes_{\mathcal{O}_U}\Omega^\bullet_U))$ has a Hopf algebra structure and a length filtration $L^N$, which is the left hand side of the \cref{eq:pi1-deRham_family} in \cref{thm:pi1-deRham_family}.

On the other hands, the right hand side of the \cref{eq:pi1-deRham_family} is written by the set of homotopy classes $\pi_1(U;a,b)$ for fixed $a,b\in U$.

\begin{dfn}
    For a topological space $T$ and two points $t_0,t_1\in T$, Let $\pi_1(T;t_0,t_1):=\{\gamma:[0,1]\to T\,|\,\text{piecewise smooth,} \gamma(0)=t_0,\gamma(1)=t_1\}/\text{(homotop)}$ the set of homotopy classes. Then, $\pi_1(T;t_0):=\pi_1(T;t_0,t_0)$ is a group structure with product determined by the path connection and $\pi_1(T;t_0,t_1)$ is a left $\pi_1(T;t_0)$-set by $\alpha\cdot\gamma:=\alpha\gamma$ for $\alpha\in\pi_1(T;t_0)$ and $\gamma\in\pi_1(T;t_0,T_1)$.
\end{dfn}

Then, we have a Hopf algebra structure on the right hand side of \cref{eq:pi1-deRham_family} if we fix an homotopy class $\gamma_0\in \pi_1(U;a,b)$ as follows. Note that
\begin{align}
    &\left((\Sym^{\le M} H_1(E_b,\bQ))^{\otimes \le N}\otimes_\bQ\bQ[\pi_1(U;a,b)]/J^{N+1}\bQ[\pi_1(U;a,b)]\otimes_\bQ\bC\right)^\vee\\
    &\cong (\Sym^{\le M} H^1(E_b,\bQ))^{\otimes \le N}\otimes_\bQ\left(\bQ[\pi_1(U;a,b)]/J^{N+1}\bQ[\pi_1(U;a,b)]\otimes_\bQ\bC\right)^\vee
\end{align}
since each $(\Sym^{\le M} H^1(E_b,\bQ))^{\otimes \le N}$ and $\bQ[\pi_1(U;a,b)]/J^{N+1}\bQ[\pi_1(U;a,b)]$ is finite dimensional. 
\begin{enumerate}
    \item (Product) It is defined diagonally by using the tensor product of $(\Sym^{\le M} H^1(E_b,\bQ))^{\otimes \le N}$ and the dual of the coproduct of $\bQ[\pi_1(U;a,b)]/J^{N+1}\bQ[\pi_1(U;a,b)]$, which is defined by $$\Delta(\gamma)=\gamma\otimes\gamma$$ for $\gamma\in \pi_1(U;a,b)$.
    \item (Coproduct) It is defined diagonally by using the coproduct of $(\Sym^{\le M} H^1(E_b,\bQ))^{\otimes \le N}$, which is defined by 
    \begin{align}
        &\Delta([\phi_{1,1}\cdots\phi_{1,k_1}|\cdots|\phi_{r,1}\cdots\phi_{r,k_r}])\\
        &:=\sum_{i=0}^r[\phi_{1,1}\cdots\phi_{1,k_1}|\cdots|\phi_{i,1}\cdots\phi_{i,k_r}]\otimes[\phi_{i+1,1}\cdots\phi_{i+1,k_{i+1}}|\cdots|\phi_{r,1}\cdots\phi_{r,k_r}],
    \end{align}
    and the dual of the product of $\bQ[\pi_1(U;a,b)]/J^{N+1}\bQ[\pi_1(U;a,b)]$.
    \item (Antipode) It is defined diagonally by using the antipode of $(\Sym^{\le M} H^1(E_b,\bQ))^{\otimes \le N}$, which is defined by $$\Delta([\phi_{1,1}\cdots\phi_{1,k_1}|\cdots|\phi_{r,1}\cdots\phi_{r,k_r}]):=(-1)^r[\phi_{r,1}\cdots\phi_{r,k_r}|\cdots|\phi_{1,1}\cdots\phi_{1,k_1}],$$ and the dual of the antipode of $\bQ[\pi_1(U;a,b)]/J^{N+1}\bQ[\pi_1(U;a,b)]$, which is defined by $\Delta(\gamma):=\gamma_0\gamma^{-1}\gamma_0$ for any $\gamma\in \pi_1(U;a,b)$.
\end{enumerate}

Now, we give the isomorphism of the \cref{eq:pi1-deRham_family}.

\begin{prop}
    For any $N,M\in\bZge{0}$ and $a,b\in U$, let
    \begin{align}
        \langle\ast,\ast\rangle: & L^{-N}\mathbb{H}^0(U, B^\bullet(\Sym^{\le M} \mathcal{H}^1(\cE|_U/U)\otimes_{\mathcal{O}_U}\Omega^\bullet_U))\\
        &\to
        \left((\Sym^{\le M} H_1(E_b,\bQ))^{\otimes \le N}\otimes_\bQ\bQ[\pi_1(U;a,b)]/J^{N+1}\bQ[\pi_1(U;a,b)]\otimes_\bQ\bC\right)^\vee
    \end{align}
    be the map determined by
    \begin{align} \label{eq:pairing}
        \left\langle[\bdeta_1;\omega_1|\cdots|\bdeta_r;\omega_r], [\widetilde{\bdsigma}_1|\cdots|\widetilde{\bdsigma}_r]\otimes\gamma\right\rangle
    \end{align}
    for $\eta_{i,j}\in \Gamma(U,\mathcal{H}^1(\cE|_U/U))$, $\sigma_{i,j}\in H_1(E_b,\bQ)$, $\omega_i\in H^1_\dR(U)$, and $\gamma:[0,1]\to U$. Then, the pairing is well-defined as a homomorphism of Hopf algebras.
\end{prop}

\begin{proof}
    The homomorphism property is clear from \cref{prop:properties_of_iterated_integral}. The well-definedness is shown by induction on $N$ as follows. The $N=0$ case can be proved from $\langle[],[]\otimes\gamma\rangle=1$ for any $\gamma\in\pi_1(U;a;b)$. For $N>0$, the element $J^{N+1}\bQ[\pi_1(U;a,b)]$ can be written as $\alpha\gamma$ for some $\alpha\in J$ and $\gamma\in J^N\bQ[\pi_1(U;a,b)]$. Then, the property \ref{item:path_connection} of \cref{prop:properties_of_iterated_integral} yields
    \begin{align}
        \langle w,\bdsigma\otimes\alpha\gamma\rangle=\sum_{w=w_1w_2} \langle w_1,\bdsigma_1\otimes\alpha\rangle\langle w_2,\bdsigma_2\otimes\gamma\rangle
    \end{align}
    for $w\in L^{-N}\mathbb{H}^0(U, B^\bullet(\Sym^{\le M} \mathcal{H}^1(\cE|_U/U)\otimes_{\mathcal{O}_U}\Omega^\bullet_U))$ and $\bdsigma\in (\Sym^{\le M} H_1(E_b,\bQ))^{\otimes \le N}$. Here, we separate $\bdsigma=\bdsigma_1\bdsigma_2$ so that $\bdsigma_i$ and $w_i$ have the same length. Since each terms vanish from the induction hypothesis, it hold $\langle w,\bdsigma\otimes\alpha\gamma\rangle=0$, which follows the conclusion.
\end{proof}

The main theorem, \cref{thm:pi1-deRham_family}, asserts that the pairing (\ref{eq:pairing}) is an isomorphism.

\begin{thm}[\cref{thm:pi1-deRham_family}] \label{thm:pi1-deRham_family_again}
    For each $M, N\in\bZge{0}$ and $a,b\in U$,
    \begin{align}
        \langle\ast,\ast\rangle: &L^{-N}\mathbb{H}^0(U, B^\bullet(\Sym^{\le M} \mathcal{H}^1(\cE|_U/U)\otimes_{\mathcal{O}_U}\Omega^\bullet_U))\\
        &\overset{\sim}{\longrightarrow}
        \left((\Sym^{\le M} H_1(E_b,\bQ))^{\otimes \le N}\otimes_\bQ\bQ[\pi_1(U;a,b)]/J^{N+1}\bQ[\pi_1(U;a,b)]\otimes_\bQ\bC\right)^\vee,
    \end{align}
    where $J\subseteq \mathbb{Q}[\pi_1(U;a)]$ is the augmentation ideal.
\end{thm}

The goal of the following subsections is to prove this theorem.

\subsection{Isomorphism of hypercohomology group}

To prove Theorem \cref{thm:pi1-deRham_family_again}, we first write the cohomology of the bar complex by hypercohomology of a certain complex of sheaves on $U^N$.

Let $\aUb^\bullet$ be a cosimplicial object with
\begin{itemize}
    \item (components) $\aUb^n:=\underbrace{U\times \cdots\times U}_{n}$.
    \item (coface map) $\delta^i:\aUb^n\to \aUb^{n+1}$ ($i=0,\dots,n$) is given by $$\delta^i(s_1,\dots,s_n):=\begin{cases}
        (a,s_1,\dots,s_n), & i=0,\\
        (s_1,\dots,s_i,s_i,\dots,s_n), & 0<i<n,\\
        (s_1,\dots,s_n,b), & i=n.
    \end{cases}$$
    \item (codegeneracy map) $\sigma^i:\aUb^{n+1}\to \aUb^n$ ($i=0,\dots,n$) is given by $$\sigma^i(s_1,\dots,s_{n+1}):=(s_1,\dots,s_i,s_{i+2},\dots,s_{n+1}).$$
\end{itemize}

Let $\Sym^{\le M}\mathcal{H}^\ast_\dR(\cE^\ast|_{U^\ast}/\aUb^\ast)\otimes\Omega_{\aUb^\ast}^\bullet$ be a cosimplicial dg-algebra given by $$\Sym^{\le M}\mathcal{H}^\ast_\dR(\cE^n|_{U^n}/\aUb^n)\otimes\Omega_{\aUb^n}^\bullet:=(\Sym^{\le M}\mathcal{H}^1_\dR(\cE|_U/U)\otimes\Omega^\bullet_U)^{\otimes n}.$$ Then, the bar complex can be written by the associated total complex of $\Sym^{\le M}\mathcal{H}^\ast_\dR(\cE^\ast|_{U^\ast}/\aUb^\ast)\otimes\Omega_{\aUb^\ast}^\bullet$, which is defined in \cref{app:cosimplicial_obj}.

\begin{lem} \label{lem:barcpx_to_normalized}
    The map 
    \begin{align}
        \psi:B^\bullet(\Sym^{\le M}\mathcal{H}^1_\dR(\cE|_U/U)\otimes\Omega_U^\bullet)\to \Tot \mathcal{N}(\Sym^{\le M}\mathcal{H}^\ast_\dR(\cE^\ast|_{U^\ast}/\aUb^\ast)\otimes\Omega_{\aUb^\ast}^\bullet)
    \end{align}
    given by
    \begin{align}
        \psi([\bdeta_1;\omega_1|\cdots|\bdeta_r\omega_r]):=(-1)^{\sum_{i=1}^r(r-i)\deg(\omega_i)}[\bdeta_1;\omega_1|\cdots|\bdeta_r\omega_r]
    \end{align}
    is an isomorphism of complexes of sheaves. In addition, it sends $L^{-N}B^\bullet(\Sym^{\le M}\mathcal{H}^1_\dR(\cE|_U/U)\otimes\Omega_U\bullet)$ to $\Tot\sigma_{\le N} \mathcal{N}(\Sym^{\le M}\mathcal{H}^\ast_\dR(\cE^\ast|_{U^\ast}/\aUb^\ast)\otimes\Omega_{\aUb^\ast}^\bullet)$.
\end{lem}

\begin{proof}
    By definition, we have the equation
    \begin{align}
        B^\bullet(\Sym^{\le M}\mathcal{H}^1_\dR(\cE|_U/U)\otimes\Omega_U^\bullet)= \Tot \mathcal{N}(\Sym^{\le M}\mathcal{H}^\ast_\dR(\cE^\ast|_{U^\ast}/\aUb^\ast)\otimes\Omega_{\aUb^\ast}^\bullet)
    \end{align}
    as sets. The difference lies in the signs of the differential maps, which is resolved by $\psi$.
\end{proof}

Secondary, by using the Riemann--Hilbert correspondence
\begin{align}
    \mathcal{H}^i_\dR(\cE^n|_{U^n}/U^n)\cong R^i\widetilde{\pi}_\ast \locQ_{U^n}\otimes\mathcal{O}_{U^n}
\end{align}
for $\widetilde{\pi}=\underbrace{\pi\times\cdots\times\pi}_{n}:\cE^n|_{U^n}\to U^n$ and the de Rham theorem
\begin{align}
    \mathbb{H}^\bullet(U^n,\Omega_{U^n})\cong H^\bullet(U^n,\bQ)\otimes\bC,
\end{align}
we have the following lemma.

\begin{lem} \label{lem:deRham_to_Betti}
    There is a quasi-isomorphism
    \begin{align}
        \Sym^{\le M}\mathcal{H}^\ast_\dR(\cE^\ast|_{U^\ast}/\aUb^\ast)\otimes\Omega_{\aUb^\ast}^\bullet\to \Sym^{\le M} R^\ast\widetilde{\pi}_\ast\locQ_{\aUb^\ast}\otimes\locQ_{\aUb^\ast}\otimes\bC.
    \end{align}
\end{lem}

Next, by using \cref{lem:cosimplicial_vs_normalized} for cosimplicial object $\Sym^{\le M} R^\ast\widetilde{\pi}_\ast\locQ_{\aUb^\ast}\otimes\locQ_{\aUb^\ast}$, we have the following lemma.

\begin{lem} \label{lem:normalized_to_cosimplicial}
    We have a functorially homotopically equivalence
    \begin{align}
        \sigma_{\le N}\mathcal{N}\Sym^{\le M} R^\ast\widetilde{\pi}_\ast\locQ_{\aUb^\ast}\otimes\locQ_{\aUb^\ast}\to C^\bullet(I_N,\Sym^{\le M} R^\ast\widetilde{\pi}_\ast\locQ_{\aUb^\ast}\otimes\locQ_{\aUb^\ast}),
    \end{align}
    where $C^\bullet(I_N,\Sym^{\le M} R^\ast\widetilde{\pi}_\ast\locQ_{\aUb^\ast}\otimes\locQ_{\aUb^\ast})$ is a complex of sheaves defined in \cref{app:cosimplicial_obj}.
\end{lem}

We now consider a new complex of sheaves. First, we give more general definition, which is also used in \cref{subsec:Beilinson}. Let $Y_0,\dots,Y_N\subseteq U$ be closed subset. We write $Y=Y_0\cup\cdots\cup Y_N$. 

\begin{notn}
    The following notation will be used.
    \begin{itemize}
        \item We write $\binom{I_N}{l}:=\{I\subseteq I_N\,|\,\# I=l\}$ for $l=0,\dots,N$.
        \item For $I\subseteq I_N$, we write $I^c:=I_N\setminus I$.
        \item For $I\subseteq I_N$, we write $Y_I:=\bigcap_{i\in I}Y_i$.
        \item For $I\subseteq K\subseteq I_N$, we denote by $$d_{K,I}:\Sym R^{N-\# I}\widetilde{\pi}_\ast\locQ_{U^N}|_{Y_I} \otimes \locQ_{Y_I} \to \Sym R^{N-\# K}\widetilde{\pi}_\ast\locQ_{U^N}|_{Y_K} \otimes \locQ_{Y_K}$$ the restriction map induced by the inclusion $Y_K\subseteq Y_I$.
        \item For $K=\{i_0,\dots,i_l\}\subseteq I_N$ with $i_0<\cdots<i_l$ and $I=K\setminus\{i_j\}$ ($j=0,\dots,l$), we denote $\epsilon(I,K):=(-1)^j$. We also write $\epsilon(K):=\prod_{i\in K}(-1)^i$.
    \end{itemize}
\end{notn}

\begin{dfn}
    For $0\le l\le N$, we define a morphism of sheaves
    \begin{align}
        d:\bigoplus_{I\in \binom{I_N}{l}} \Sym R^{N-\# I}\widetilde{\pi}_\ast\locQ_{U^N}|_{Y_I} \otimes \locQ_{Y_I} \to \bigoplus_{I\in \binom{I_N}{l+1}} \Sym R^{N-\# K}\widetilde{\pi}_\ast\locQ_{U^N}|_{Y_K} \otimes \locQ_{Y_I}
    \end{align}
    as $$d=\bigoplus_{\substack{I\in\binom{I_N}{l}, K\in\binom{I_N}{l+1}\\I\subseteq K}} \epsilon(I,K)d_{K,I}.$$ We denote by $K(X;Y_0,\dots,Y_N)$ the complex of sheaves
    \begin{align}
        0\to \bigoplus_{I\in \binom{I_N}{0}} \Sym R^{N-\# I}\widetilde{\pi}_\ast\locQ_{U^N}|_{Y_I} \otimes \locQ_{Y_I} &\overset{d}{\longrightarrow} \bigoplus_{I\in \binom{I_N}{1}} \Sym R^{N-\# I}\widetilde{\pi}_\ast\locQ_{U^N}|_{Y_I} \otimes \locQ_{Y_I} \overset{d}{\longrightarrow}\\
         \cdots & \overset{d}{\longrightarrow} \bigoplus_{I\in \binom{I_N}{N}} \Sym R^{N-\# I}\widetilde{\pi}_\ast\locQ_{U^N}|_{Y_I} \otimes \locQ_{Y_I} \to 0.
    \end{align}
    We also define the complex of sheaves $\widetilde{K}(X;Y_0,\dots,Y_N)$ as
        \begin{align}
        0\to \bigoplus_{I\in \binom{I_N}{0}} \Sym R^{N-\# I}\widetilde{\pi}_\ast\locQ_{U^N}|_{Y_I} \otimes \locQ_{Y_I} &\overset{d}{\longrightarrow} \bigoplus_{I\in \binom{I_N}{1}} \Sym R^{N-\# I}\widetilde{\pi}_\ast\locQ_{U^N}|_{Y_I} \otimes \locQ_{Y_I} \overset{d}{\longrightarrow}\\
        \cdots & \overset{d}{\longrightarrow} \bigoplus_{I\in \binom{I_N}{N+1}} \Sym R^{N-\# I}\widetilde{\pi}_\ast\locQ_{U^N}|_{Y_I} \otimes \locQ_{Y_I} \to 0.
    \end{align}
\end{dfn}

We now specialize in our setting: $$Y_i:=\begin{cases}
    \{(s_1,\dots,s_N)\in U^N\,|\,s_1=a\}, & i=0,\\
    \{(s_1,\dots,s_N)\in U^N\,|\,s_i=s_{i+1}\}, & 0<i<N,\\
    \{(s_1,\dots,s_N)\in U^N\,|\,s_N=b\}, & i=N.
\end{cases}$$ We just write
\begin{align}
    \aKb{N}&:=K(U;Y_0,\dots,Y_N),\\
    \atKb{N}&:=\widetilde{K}(U;Y_0,\dots,Y_N).
\end{align}

Then we have the following lemma.
\begin{lem} \label{lem:cosimplicial_to_K}
    There exists a functorial isomorphism
    \begin{align}
        C^\bullet(I_N,\Sym^{\le M} R^\ast\widetilde{\pi}_\ast\locQ_{\aUb^\ast}\otimes\locQ_{\aUb^\ast}) \to \aKb{N} [N].
    \end{align}
\end{lem}

\begin{proof}
    For $I\subseteq I_N$, we denote by
    \begin{align}
        f_I: C^\ast(\aUb^I,\bQ) \overset{\sim}{\longrightarrow} \Sym R^\ast\widetilde{\pi}_\ast\locQ_{U^N}|_{Y_{I^c}} \otimes \locQ_{Y_{I^c}}
    \end{align}
    the isomorphism of sheaves induced from $\aUb^I=Y_{I^c}$. Then, the map
    \begin{align}
        C^\bullet(I_N,\Sym^{\le M} R^\ast\widetilde{\pi}_\ast\locQ_{\aUb^\ast}\otimes\locQ_{\aUb^\ast}) \to \aKb{N} [N]
    \end{align}
    is determined such that the direct sum components of the map are $f_I$. The compatibility with their differential maps is proved from the commutative diagram
    \begin{align}
        \xymatrix {
            \aUb^I \ar[r]^{\delta_{I,K}} \ar@{=}[d] & \aUb^K \ar@{=}[d]\\
            Y_{I^c} \ar@{^{(}->}[r] & Y_{K^c}
        }
    \end{align}
    for $K=\{i_0,\dots,i_l\}\supseteq K\setminus\{i_j\}= I$ and 
    \begin{align}
        \delta_{I,K}(x_0,\dots,x_l):=(x_0,\dots,\widehat{x_{i}},x_l).
    \end{align}
\end{proof}

Combining \cref{lem:barcpx_to_normalized,lem:deRham_to_Betti,lem:normalized_to_cosimplicial,lem:cosimplicial_to_K}, we obtain the following isomorphism
\begin{align} \label{eq:first_step}
    L^{-N}\mathbb{H}^0(B^\bullet(U,\Sym^{\le M} \mathcal{H}^1(\cE|_U/U)\otimes_{\mathcal{O}_U}\Omega^\bullet_U))\cong \mathbb{H}^N(U^N,\aKb{N})\otimes \bC.
\end{align}

Thus, to show \cref{thm:pi1-deRham_family_again}, it is sufficient to show 
\begin{align} \label{eq:second_step}
    \mathbb{H}^N(U^N,\aKb{N})\cong \left((\Sym^{\le M} H_1(E_b,\bQ))^{\otimes \le N}\otimes_\bQ\bQ[\pi_1(U;a,b)]/J^{N+1}\bQ[\pi_1(U;a,b)]\right)^\vee.
\end{align}

\begin{rem}
    Beilinson showed that for a manifold $M$, $x,y\in M$, and any $N\in\bZge{0}$,
    \begin{align}
        \mathbb{H}^N(M^N,K(M^N;Y_0,\dots,Y_N))\cong (\bQ[\pi_1(M;x,y)]/J^{N+1}\bQ[\pi_1(M;x,y)])^\vee,
    \end{align}
    where $Y_i:=\{(x_1,\dots,x_N)\in S^N \,|\, x_i=x_{i+1}\}$ ($x_0:=x, x_{N+1}:=y$) (see \cite[Theorem 3.307, p.264]{GiFr}). Hence, we will discuss an analogue of this result in the case of the Legendre family in the next section.
\end{rem}

\subsection{Beilinson's theorem for Legendre family} \label{subsec:Beilinson}
We now prove the \cref{eq:second_step}. First, we give a homomorphism
\begin{align}
    \varphi_b: \mathbb{H}^N(U^N,\aKb{N})\to \left((\Sym^{\le M} H_1(E_b,\bQ))^{\otimes \le N}\otimes_\bQ\bQ[\pi_1(U;a,b)]/J^{N+1}\bQ[\pi_1(U;a,b)]\right)^\vee.
\end{align}

For $(\eta,\omega)=(\eta_I,\omega_I)_I\in \mathbb{H}^N(U^N,\aKb{N})$ ($\eta_I \in \Gamma(U,R^{N-\# I}\widetilde{\pi}_\ast\locQ_{Y_I})$, $\omega_I\in H^{N-\# I}(Y_I,\bQ)$), $\bdsigma\in (\Sym^{\le M} H_1(E_b,\bQ))^{\otimes \le N}$, and $\gamma\in\pi_1(U;a,b)$, $\eta_I(\widetilde{\bdsigma})\omega_I$ is a $(N-\# I)$-cocycle on $Y_I$. Then, we define 
\begin{align}
    \varphi_b(\eta,\omega)(\bdsigma\otimes\gamma):=\sum_{I\subseteq I_N}(-1)^\epsilon\langle \eta_I(\widetilde{\bdsigma})\omega_I, \sigma_{q,\gamma}^{N,I}\rangle \in \bQ,
\end{align}
where
\begin{align}
    \epsilon:=(-1)^{\binom{N-\# I}{2}}(-1)^{N\# I^c}\epsilon(I^c)
\end{align}
and $\sigma_{b,\gamma}^{N,I}$ is a $(N-\# I)$-cycle on $Y_I\cong U^{N-\# I}$ determined by
\begin{align}
    \sigma_{b,\gamma}^{N,I}(t_1,\dots,t_{N-\# I}):=(\gamma(t_1),\dots,\gamma(t_{N-\# I})).
\end{align}

The goal is to show that $\varphi_b$ is isomorphic. To prove this, we consider a relative version of $\varphi_b$ in which $a$ is fixed and $b$ is varied. We consider $U^{N,1}:=U^N\times U$ over $U$ with a projection $\varpi: U^{N,1}\to U$ to the last component. For $i=0,\dots,N$, we define the closed subset
\begin{align}
    Z_i:=\{(s_1,\dots,s_N,s_{N+1})\in U^{N,1} \,|\, s_i=s_{i+1}\},
\end{align}
where $s_0:=a$. We also white $Z_I:=\bigcap_{i\in I}Z_i$. We now consider the complex of sheaves
\begin{align}
    \aK{N}:=K(U^{N,1};Z_0,\dots,Z_N),\\
    \atK{N}:=\widetilde{K}(U^{N,1};Z_0,\dots,Z_N)
\end{align}
on $S^{N,1}$, which are relative versions of $\aKb{N}$ and $\atKb{N}$; that is,
\begin{align}
    \aK{N}|_{\varpi^{-1}(b)}=\aKb{K},\\
    \atK{N}|_{\varpi^{-1}(b)}=\atKb{K},
\end{align}
respectively.

Let $\iota_b:U^{N-1}\to U^N$ be a map $$\iota(s_1,\dots,s_{N-1}):=(s_1,\dots,s_{N-1},b).$$
Then, for each $N\ge1$, we have an exact sequence of complexes
\begin{align}\label{eq:exact_seq_iota}
    0\to (\iota_b)_\ast \aKb{N-1}[-1] \to \aKb{N} \to \atK{N-1} \to 0,
\end{align}
which is induced by the exact sequence of sheaves
\begin{align}
    0 \to \bigoplus_{I\in\binom{I_{N-1}}{l-1}} \Sym R^{N-\# I - 1}\widetilde{\pi}_\ast\locQ_{U^N}|_{Y_{I\cup \{N\}}} \otimes \locQ_{Y_{I\cup\{N\}}} &\to \bigoplus_{I\in \binom{I_N}{l}}\Sym R^{N-\# I}\widetilde{\pi}_\ast\locQ_{U^N}|_{Y_I} \otimes \locQ_{Y_I}\\
    &\to \bigoplus_{I\in \binom{I_{N-1}}{l}}\Sym R^{N-\# I}\widetilde{\pi}_\ast\locQ_{U^N}|_{Z_I} \otimes \locQ_{Z_I} \to 0.   
\end{align}
We also have an exact sequence
\begin{align} \label{eq:exact_seq_pKb_and_ptKb}
    0 \to \locQ_{(\underbrace{a,\dots,a}_{N})}[-N] \to \atK{N-1} \to \aK{N-1} \to 0,
\end{align}
where $\locQ_{(\underbrace{a,\dots,a}_{N})}$ is the skyscraper sheaf at $(a,\dots,a)\in U^{N-1,1}$.

Next, we define a relative version of $\mathbb{Q}[\pi_1(U;a,b)]$, which is a local system determined by the representation $\rho: \pi_1(U;a) \curvearrowright \mathbb{Q}[\pi_1(U;a,b)]$ defined by $\rho(\alpha)(\gamma):=\alpha\gamma$. This local system is denoted by $\mathbb{Q}[\pi_1(U;a,\bullet)]$. Then, the right hand side of the \cref{eq:second_step} is the stalk of the local system
\begin{align}
    \left((\Sym^{\le M} R_1\pi_\ast\locQ_U)^{\otimes \le N}\otimes_\bQ\bQ[\pi_1(U;a,\bullet)]/J^{N+1}\bQ[\pi_1(U;a,\bullet)]\right)^\vee
\end{align}
at $b$. We simply write the above local system as $\mathcal{L}^N$.

\begin{lem}
    There exists a morphism of local systems
    \begin{align}
        \varphi: R^N\varpi_\ast(\aK{N}) \to \mathcal{L}^N
    \end{align}
    such that the induced map between their stalks is
    \begin{align}
        \varphi_b: \mathbb{H}^N(U^N,\aKb{N})\to (\mathcal{L}^N)_b.
    \end{align}
\end{lem}

\begin{proof}
    We have to show that each maps $\varphi_b$ grue together to a morphism of local systems, that is, for a contractible open subset $V\subseteq U$ and two points $b,b'\in V$, the diagram
    \begin{align}
        \xymatrix {
            \mathbb{H}^N(U^N,\aKb{N}) \ar[d]^{\varphi_b} & R^N\varpi_\ast(\aK{N})(V) \ar[l]_{\sim} \ar[r]^{\sim} & \mathbb{H}^N(U^N,\aKbp{N}) \ar[d]_{\varphi_{b'}}\\
            (\mathcal{L}^N)_b & \mathcal{L}^N(V) \ar[l]_{\sim} \ar[r]^{\sim} & (\mathcal{L}^N)_{b'}
        }
    \end{align}
    is commutative. That follows from
    \begin{align}
        \varphi_b(\eta,\omega)(\bdsigma\otimes\gamma)=\varphi_{b'}(\eta,\omega)(\bdsigma\otimes\gamma\gamma')
    \end{align}
    for any path $\gamma'$ from $b$ to $b'$.
\end{proof}

Then, it is enough to show that $\varphi$ is isomorphic. We will show it by induction on $N$.

\begin{lem} \label{lem:about_varphi}
    Under the assumption of the induction hypothesis, we have the followings.
    \begin{enumerate}[label=(\roman*)]
        \item It holds $R^i\varpi_\ast(\aK{N})=0$ for any $i\le N-1$.
        \item $\varphi$ is isomorphic. \label{item:map_varphi_is_isom}
    \end{enumerate}
\end{lem}

\begin{proof}

    Since the exact sequence (\ref{eq:exact_seq_iota}), we have a long exact sequence

    \hspace{-5.7cm}
    \begin{tikzpicture}

        \node (xshift) at (4,0) {};
        \node (yshift) at (0,1.2) {};

        \node (0) at ($1*(xshift)+2*(yshift)$) {$\cdots$};
        \node (1) at ($2*(xshift)+2*(yshift)$) {$\mathbb{H}^{N-1}(U^{N-1,1},\atK{N-1})$};
        \node (2) at ($0*(xshift)+1*(yshift)$) {$\mathbb{H}^{N-1}(U^{N-1},\aKb{N-1})$};
        \node (3) at ($1*(xshift)+1*(yshift)$) {$\mathbb{H}^{N}(U^N,\aKb{N})$};
        \node (4) at ($2*(xshift)+1*(yshift)$) {$\mathbb{H}^{N}(U^{N-1,1},\atK{N-1})$};
        \node (5) at ($0*(xshift)+0*(yshift)$) {$\mathbb{H}^{N}(U^{N-1},\aKb{N-1})$};
        \node (6) at ($1*(xshift)+0*(yshift)$) {$\cdots$};
        
        \draw[->] ($(0)+(+1.3,0)$) to (1);
        \draw[->] (2) to (3);
        \draw[->] (3) to (4);
        \draw[->] (5) to ($(6)+(-1.5,0)$);

        \draw (1) to[
            out=0,
            in=180,
            looseness=1.35
        ] (2);
        \draw (4) to[
            out=0,
            in=180,
            looseness=1.35
        ] node[below right] {$f$} (5);

        \useasboundingbox
            ([xshift=+5cm]2.west)
            rectangle
            ([xshift=-5cm]4.east);

    \end{tikzpicture}

    Then, if we put $V:=\Sym^{\le M} H_1(E_b,\bQ)$, we deduce the following diagram
    \begin{align} \label{eq:diagram_of_kerf}
        \xymatrix {
            & 0 \ar[d]\\
            \mathbb{H}^{N-1}(U^{N-1},\aKb{N-1}) \ar[d] \ar[r]_{\varphi_b}^{\sim} & (\mathcal{L}^{N-1})_b \ar[d]\\
            \mathbb{H}^{N}(U^N,\aKb{N}) \ar[d] \ar[r]_{\varphi_b} & (\mathcal{L}^N)_b \ar[d]\\
            \Ker f \ar[d] \ar@{-->}[r] & \left(V^{\otimes \le N-1}\otimes_\bQ (J^N/ J^{N+1}) \oplus V^{\otimes N}\otimes_\bQ \bQ[\pi_1(U;a,b)]/J^{N+1}\right)^\vee \ar[d]\\
            0 & 0
        }
    \end{align}
    in which the first and second columns are exact. Here, $J^N\bQ[\pi_1(U;a,b)]$ and $J^{N+1}\bQ[\pi_1(U;a,b)]$ were abbreviated as $J^N$ and $J^{N+1}$, respectively. Since the above square is commutative, the map
    \begin{align} \label{eq:map_from_kerf}
        \varphi_b: \Ker f \to \left((\Sym^{\le M} R_1\pi_\ast\locQ_U)^{\otimes \le N}\otimes_\bQ (J^N/ J^{N+1})\right)^\vee
    \end{align}
    is obtained which makes the below square commutative.


    \begin{sublem} \label{sublem:about_map_from_kerf}
        Under the assumption of the induction hypothesis, we have the followings
        \begin{enumerate}[label=(\roman*)]
            \item It holds $\mathbb{H}^i(U^{N-1,1},\atK{N-1})=0$ for any $i\le N-1$. \label{item:sublem_hypercohomology_vanish}
            \item The map (\ref{eq:map_from_kerf}) is isomorphic. \label{item:sublem_map_from_kerf_is_isom}
        \end{enumerate}
    \end{sublem}

    If \cref{sublem:about_map_from_kerf} is proved, then \cref{lem:about_varphi} is valid due to five lemma for the diagram (\ref{eq:diagram_of_kerf}). Hence, we prove \cref{sublem:about_map_from_kerf} below.

    We first prove \ref{item:sublem_hypercohomology_vanish} of \cref{sublem:about_map_from_kerf} by induction on $i$. Let $E_r^{p,q}$ be the Leray spectral sequence associated with $\varpi$, that is
    \begin{align}
        E^{p,q}_2=H^p(U,R^q\varpi_\ast(\atK{N-1})) \Longrightarrow \mathbb{H}^{p+q}(U^{N-1,1},\atK{N-1}).
    \end{align}

    From the exact sequence (\ref{eq:exact_seq_pKb_and_ptKb}) and the induction hypothesis, we have
    \begin{align}
        R^i\varpi_\ast(\atKb{N-1}) \cong R^i\varpi_\ast(\aKb{N-1})=0
    \end{align}
    for $i\le N-2$ and an exact sequence
    \begin{align} \label{eq:exact_seq_skyscraper}
        0\to R^{N-1}\varpi_\ast(\atK{N-1})\to R^{N-1}\varpi_\ast(\aK{N-1}) \to \underline{(\Sym^{\le M}H_1(E_a,\bQ))^{\otimes\le N-1}}_a \to 0.
    \end{align}

    Therefore, the second page of the spectral sequence $E_2^{p,q}$ vanishes if $q\le N-2$, and from this, we have an equality
    \begin{align}
        \mathbb{H}^i(S^{N-1,1},\atK{N-1})=\begin{cases}
            0, & i\le N-2,\\
            H^0(S,R^{N-1}\varpi_\ast(\atK{N-1})), & i=N-1
        \end{cases}
    \end{align}
    and an exact sequence
    \begin{align} \label{eq:exact_seq_from_spec_seq_of_varpi}
        0\to H^1(U,R^{N-1}\varpi_\ast(\atK{N-1}))\to \mathbb{H}^N(S^{N-1,1}\atK{N-1})\to H^0(U,R^N\varpi_\ast(\atK{N-1}))\to0.
    \end{align}

    It remains to show $H^0(U,R^{N-1}\varpi_\ast(\atK{N-1}))=0$ to prove \ref{item:sublem_hypercohomology_vanish} of \cref{sublem:about_map_from_kerf}. By taking the cohomologies of the exact sequence (\ref{eq:exact_seq_skyscraper}), we deduce
    \begin{align} \label{eq:exact_seq_from_exact_seq_skyscraper}
        0&\to H^0(U,R^{N-1}\varpi_\ast(\atK{N-1})) \to H^0(U,R^{N-1}\varpi_\ast(\aK{N-1})) \overset{g}{\longrightarrow} (\Sym^{\le M}H_1(E_a,\bQ))^{\otimes\le N-1}\\
        &\to H^1(U,R^{N-1}\varpi_\ast(\atK{N-1})) \to H^1(U,R^{N-1}\varpi_\ast(\aK{N-1})) \to 0.
    \end{align}

    To show $H^0(U,R^{N-1}\varpi_\ast(\atK{N-1}))=0$, we will prove $g$ in the above exact sequence is injective. By the induction hypothesis, it holds $$(R^{N-1}\varpi_\ast(\aK{N-1}))_a\cong (\mathcal{L}^{N-1})_a=\left((\Sym^{\le M}H_1(E_a,\bQ))^{\otimes\le N-1}\otimes_\bQ\bQ[\pi_1(U;a)]/J^N\right)^\vee$$ and then, we have
    \begin{align} \label{eq:cohomology_of_local_system_and_group_cohomology}
        H^i(U,R^{N-1}\varpi_\ast(\aK{N-1}))\cong H^i\left(\pi_1(U;a),\left((\Sym^{\le M}H_1(E_a,\bQ))^{\otimes\le N-1}\otimes_\bQ\bQ[\pi_1(U;a)]/J^N\right)^\vee\right)
    \end{align}
    for $i=0,1$. Also, from the short exact sequence
    \begin{align}
        0&\to \left((\Sym^{\le M}H_1(E_a,\bQ))^{\otimes\le N-1}\otimes_\bQ\bQ[\pi_1(U;a)]/J^N\right)^\vee\\
        &\to \left((\Sym^{\le M}H_1(E_a,\bQ))^{\otimes\le N-1}\otimes_\bQ\bQ[\pi_1(U;a)]\right)^\vee\\
        &\to \left((\Sym^{\le M}H_1(E_a,\bQ))^{\otimes\le N-1}\otimes_\bQ J^N\right)^\vee \to 0,
    \end{align}
    we have a long exact sequence
    \begin{align} \label{eq:exact_seq_group_cohomology}
        0 &\longrightarrow H^0\left(\pi_1(U;a),\left((\Sym^{\le M}H_1(E_a,\bQ))^{\otimes\le N-1}\otimes_\bQ\bQ[\pi_1(U;a)]/J^N\right)^\vee\right)\\
        &\longrightarrow H^0\left(\pi_1(U;a),\left((\Sym^{\le M}H_1(E_a,\bQ))^{\otimes\le N-1}\otimes_\bQ\bQ[\pi_1(U;a)]\right)^\vee\right)\\
        &\overset{h}{\longrightarrow} H^0\left(\pi_1(U;a),\left((\Sym^{\le M}H_1(E_a,\bQ))^{\otimes\le N-1}\otimes_\bQ J^N\right)^\vee\right)\\
        &\longrightarrow H^1\left(\pi_1(U;a),\left((\Sym^{\le M}H_1(E_a,\bQ))^{\otimes\le N-1}\otimes_\bQ\bQ[\pi_1(U;a)]/J^N\right)^\vee\right) \to 0.
    \end{align}

    Since the fundamental group $\pi_1(U;a)$ acts diagonally and the invariant part of the action $\pi_1(U;a)\curvearrowright \bQ[\pi_1(U;a)]^\vee$ is $\bQ$, we have
    \begin{align}
        H^0\left(\pi_1(U;a),\left((\Sym^{\le M}H_1(E_a,\bQ))^{\otimes\le N-1}\otimes_\bQ\bQ[\pi_1(U;a)]\right)^\vee\right)\subseteq (\Sym^{\le M}H_1(E_a,\bQ))^{\otimes\le N-1}.
    \end{align}

    Then, It was shown that $g$ is injective through the isomorphism (\ref{eq:cohomology_of_local_system_and_group_cohomology}) and we have
    \begin{align}
        H^0(S,R^{N-1}\varpi_\ast(\atK{N-1}))=0,
    \end{align}
    which implies \ref{item:sublem_hypercohomology_vanish} of \cref{sublem:about_map_from_kerf}.

    We now prove \ref{item:sublem_map_from_kerf_is_isom} of \cref{sublem:about_map_from_kerf}. By computing the invariant part of the $\pi_1(U;a)$-action, it holds
    \begin{align}
        &H^0\left(\pi_1(U;a),\left((\Sym^{\le M}H_1(E_a,\bQ))^{\otimes\le N-1}\otimes_\bQ J^N\right)^\vee\right)\\
        &\cong \left((\Sym^{\le M}H_1(E_a,\bQ))^{\otimes\le N-1}\otimes_\bQ J^N/J^{N+1} \oplus (\Sym^{\le M}H_1(E_a,\bQ))^{\otimes N}\otimes_\bQ \bQ[\pi_1(U;a)]/J^{N+1} \right)^\vee.
    \end{align}
    Then, from the \cref{eq:cohomology_of_local_system_and_group_cohomology}, and the map $h$ in the exact sequence (\ref{eq:exact_seq_group_cohomology}) is zero map, we deduce
    \begin{align}
        &H^1(U,R^{N-1}\varpi_\ast(\aK{N-1}))\\
        &\cong \left((\Sym^{\le M}H_1(E_a,\bQ))^{\otimes\le N-1}\otimes_\bQ J^N/J^{N+1} \oplus (\Sym^{\le M}H_1(E_a,\bQ))^{\otimes N}\otimes_\bQ \bQ[\pi_1(U;a)]/J^{N+1} \right)^\vee.
    \end{align}
    Combining the above discussion with the diagram (\ref{eq:diagram_of_kerf}), it is sufficient to show
    \begin{align}
        \Ker f= H^1(U,R^{N-1}\varpi_\ast(\aK{N-1}))
    \end{align}
    to prove \ref{item:sublem_map_from_kerf_is_isom}. For this, we consider the long exact sequence
    \begin{align}
        \cdots\to R^{N-1}\varpi_\ast(\iota_b)_\ast(\aKb{N-1})\overset{\psi}{\longrightarrow} R^N\varpi_\ast(\aKb{N})\to R^N\varpi_\ast(\atK{N-1})\to\cdots
    \end{align}
    deduced from the exact sequence (\ref{eq:exact_seq_iota}). On the other hands, by considering the Leray spectral sequence associated with $\varpi\circ \iota_b$, we obtain an exact sequence
    \begin{align} \label{eq:exact_seq_from_spec_seq_of_varpi_iota}
        0\to H^1(U,R^{N-1}\varpi_\ast(\iota_b)_\ast(\aKb{N-1}))&\to \mathbb{H}^N(U^{N-1},\aKb{n-1})\\
        &\to H^0(U,R^N\varpi_\ast(\iota_b)_\ast(\aKb{N-1}))\to0
    \end{align}
    as same as the exact sequence (\ref{eq:exact_seq_from_spec_seq_of_varpi}). Since $R^{N-1}\varpi_\ast(\iota_b)_\ast(\aKb{N-1})$ is a skyscraper sheaf, it hold
    \begin{align}
        H^1(U,R^{N-1}\varpi_\ast(\iota_b)_\ast(\aKb{N-1}))=0
    \end{align}
    and the right map of the above exact sequence.

    Let $\mathcal{F}:=\Coker\psi$ be a sheaf on $U$ and consider the commutative diagram
    \begin{align} \label{eq:diagram_from_spec_seq}
        \xymatrix {
            & 0 \ar[d] & \\
            & H^1(U,R^{N-1}\varpi_\ast(\atK{N-1})) \ar[d] \ar[r] & 0 \ar[d]\\
            & \mathbb{H}^{N}(U^{N-1,1},\aK{N-1}) \ar[d] \ar[r]^{f} & \mathbb{H}^N(U^{N-1},\aKb{N-1}) \ar[d]\\
            H^0(U,\mathcal{F}) \ar[r] & H^0(U,R^N\varpi_\ast(\atK{N-1})) \ar[d] \ar[r] & H^0(U,R^N\varpi_\ast(\iota_y)_\ast(\aKb{N-1})) \ar[d]\\
            & 0 & 0
        },
    \end{align}
    whose first and second columns are the exact sequences (\ref{eq:exact_seq_from_spec_seq_of_varpi}) and \cref{eq:exact_seq_from_spec_seq_of_varpi_iota}, respectively. The above square implies
    \begin{align}
        \Ker f\supseteq H^1(U,R^{N-1}\varpi_\ast(\aK{N-1}))
    \end{align}
    and the equation
    \begin{align}
        \Ker f= H^1(U,R^{N-1}\varpi_\ast(\aK{N-1}))
    \end{align}
    is derived from $H^0(U,\mathcal{F})=0$ and a diagram chase for the above diagram. $H^0(U,\mathcal{F})=0$ is proved as follows for the cases $a\ne b$ and $a=b$, respectively.

    \begin{enumerate}
        \item (Case $a\ne b$.)
            Set $V:=S^{N-1,1}\setminus\varpi^{-1}(b)$. Then, from the exact sequence (\ref{eq:exact_seq_iota}), it holds
            \begin{align}
                \aKb{N}|_V\cong \atK{N-1}|_V.
            \end{align}
            Also, since the map
            \begin{align}
                R^{N-1}\varpi_\ast(\aK{N-1}|_V) \to R^N\varpi_\ast\underline{(\Sym^{\le M}H_1(E_a,\bQ))^{\otimes\le N-1}}_{(a,\dots,a)}[-N]
            \end{align}
            in the long exact sequence
            \begin{align}
                &R^{N-1}\varpi_\ast(\aK{N-1}|_V) \to R^N\varpi_\ast\underline{(\Sym^{\le M}H_1(E_a,\bQ))^{\otimes\le N-1}}_{(a,\dots,a)}[-N]\\
                &\to R^{N}\varpi_\ast(\aK{N}|_V) \to R^{N}\varpi_\ast(\aK{N-1}|_V) \to 0
            \end{align}
            is surjective, it holds
            \begin{align}
                R^{N}\varpi_\ast(\aK{N})|_{U\setminus\{b\}} \to R^{N}\varpi_\ast(\aK{N-1})|_{U\setminus\{b\}}.
            \end{align}
            Thus, for any contractible open subset $W\subseteq V$ with $a\notin W$ and $b\in W$, we have
            \begin{align}
                \mathbb{H}^N(\varpi^{-1}(W),\aK{N-1})\cong \mathbb{H}^N(\varpi^{-1}(W),\aKb{N-1})
            \end{align}
            and
            \begin{align}
                R^N\varpi_\ast(\aKb{N-1})_b=0,
            \end{align}
            which implies $\mathcal{F}=0$.

        \item (Case $a=b$.)
            Set $V:=S^{N-1,1}\setminus\varpi^{-1}(b)$. Then, from the exact sequence (\ref{eq:exact_seq_iota}), it holds
            \begin{align}
                \aKa{N}|_V\cong \atK{N-1}|_V.
            \end{align}
            Since $\mathcal{F}|_{U\setminus\{a\}}=R^N\varpi_\ast(\aK{N})|_{U\setminus\{a\}}$ is a local system, we have
            \begin{align}
                \mathbb{H}^N(\varpi^{-1}(W),\aKa{N})=\mathbb{H}^0(\varpi^{-1}(W),\underline{(\Sym^{\le M}H_1(E_a,\bQ))^{\otimes\le N-1}}_{(a,\dots,a)})
            \end{align}
            and the stalk of $R^N\varpi_\ast(\aK{N})$ at $a$ is
            \begin{align}
                R^N\varpi_\ast(\aK{N})_a=(\Sym^{\le M}H_1(E_a,\bQ))^{\otimes\le N-1}.
            \end{align}
            In that case, the map
            \begin{align}
                \psi_a: R^{N-1}\varpi_\ast(\iota_b)_\ast(\aKb{N-1})_a \to R^N\varpi_\ast(\aKb{N})_a
            \end{align}
            between stalks is surjective, which implies $\mathcal{F}=0$.
    \end{enumerate}

    Therefore, in each cases $a\ne b$ and $a=b$, it holds $\mathcal{F}=0$ and we deduce \ref{item:sublem_map_from_kerf_is_isom} of \cref{sublem:about_map_from_kerf} and also \cref{lem:about_varphi}.

\end{proof}


\subsection{Tangential base point and regularized iterated integral} \label{subsec:regularized_iterated_integral}

We can extend \cref{thm:pi1-deRham_family} to the case of regularized iterated integrals.

\begin{thm}[\cref{thm:pi1-deRham_family_reg}] \label{thm:pi1-deRham_family_reg_again}
    For each $M, N\in\bZge{0}$ and two tangential base points $\mathbf{a}, \mathbf{b}$ on $S$, regularized iterated integral induces an isomorphism
    \begin{align}
        &L^{-N}\mathbb{H}^0(S, B^\bullet(\Sym^{\le M} \mathcal{H}^1(\cE/S)\otimes_{\mathcal{O}_S}\Omega^\bullet_S))\\
        &\cong
        \left((\Sym^{\le M} H_1(E_\mathbf{b},\bQ))^{\otimes \le N}\otimes_\bQ\bQ[\pi_1(S;\mathbf{a},\mathbf{b})]/J^{N+1}\bQ[\pi_1(S;\mathbf{a},\mathbf{b})]\otimes_\bQ\bC\right)^\vee,
    \end{align}
    where $J\subseteq \mathbb{Q}[\pi_1(S;\mathbf{a})]$ is the augmentation ideal.
\end{thm}

\begin{proof}
    The well-definedness of the pairing follows from the well-definedness of regularized iterated integral and \cref{prop:properties_of_regularized_iterated_integral}. So, we prove that the pairing is non-degenerate, that is, there is no non-zero element $\bdsigma\otimes\gamma\in (\Sym^{\le M} H_1(E_\mathbf{b},\bQ))^{\otimes \le N}\otimes_\bQ\bQ[\pi_1(S;\mathbf{a},\mathbf{b})]/J^{N+1}\bQ[\pi_1(S;\mathbf{a},\mathbf{b})]$ such that $\langle w,\bdsigma\otimes\gamma\rangle$ for all $w\in L^{-N}\mathbb{H}^0(S, B^\bullet(\Sym^{\le M} \mathcal{H}^1(\cE/S)\otimes_{\mathcal{O}_S}\Omega^\bullet_S))$. Assume that such a $\bdsigma\otimes\gamma\in (\Sym^{\le M} H_1(E_\mathbf{b},\bQ))^{\otimes \le N}\otimes_\bQ\bQ[\pi_1(S;\mathbf{a},\mathbf{b})]/J^{N+1}\bQ[\pi_1(S;\mathbf{a},\mathbf{b})]$ exists. We take $a',b'\in S$ and two paths; $\gamma_1$ from $a'$ to $\mathbf{a}$ and $\gamma_2$ from $\mathbf{b}$ to $b'$. Then, for any $w\in L^{-N}\mathbb{H}^0(S, B^\bullet(\Sym^{\le M} \mathcal{H}^1(\cE/S)\otimes_{\mathcal{O}_S}\Omega^\bullet_S))$, it holds
    \begin{align}
        \langle w,\bdsigma\otimes\gamma_1\gamma\gamma_2\rangle=\sum_{w=w_1w_2w_3}\langle w_1,\bdsigma_1\otimes\gamma_1 \rangle \langle w_2,\bdsigma_2\otimes\gamma \rangle \langle w_3,\bdsigma_3\otimes\gamma_2 \rangle.
    \end{align}
    Here, we separate $\bdsigma=\bdsigma_1\bdsigma_2\bdsigma_3$ so that $\bdsigma_i$ and $w_i$ have the same length. Then, since the length of $\bdsigma_2$ is less than or equal to $N$, we have
    \begin{align}
        \langle w,\bdsigma\otimes\gamma_1\gamma\gamma_2\rangle=0.
    \end{align}
    Note that $\gamma_1\gamma\gamma_2$ is a usual (not a tangential) path. Therefore, according to \cref{thm:pi1-deRham_family}, it holds $\bdsigma\otimes\gamma_1\gamma\gamma_2=0$, which implies $\bdsigma\otimes\gamma=0$. 
\end{proof}

\appendix
\crefalias{section}{appendix}
\section{Cosimplicial object} \label{app:cosimplicial_obj}

In this section, we give a properties of cosimplicial objects, which is used in the proof of the main theorem.

Let $\mathcal{I}$ be the small category whose objects are finite ordered sets $I_n$ ($n\in\bZge{0}$) and whose morphisms are non-decreasing maps $I_n\to I_m$ for any $n,m\in\bZge{0}$. We write the special morphisms; cofaces $\delta^i:I_n\to I_{n+1}$ ($i=0,\dots,n+1$) and codegeneracies $\sigma^i:I_{n+1}\to I_n$ ($i=0,\dots,n$) defined by
\begin{align}
    \delta^i(j):=\begin{cases}
        j, & j<i,\\
        j+1, & j\ge i,
    \end{cases}, &&
    \sigma^i(j):=\begin{cases}
        j, & j\le i,\\
        j-1, & j> i.
    \end{cases}
\end{align}
Note that any morphism in $\mathcal{I}$ can be written as a composition of cofaces and codegeneracies.

\begin{dfn}
    Let $\mathcal{C}$ be a category. A cosimplicial object of $\mathcal{C}$ is a functor $\mathcal{I}\to\mathcal{C}$. We denote $X^\bullet:I_n\rightsquigarrow X^n$ for a cosimplicial object $X^\bullet$. The image of a coface and a codegeneracy under the functor is also called a coface and a codegeneracy, respectively.
\end{dfn}

\begin{ex}
    For a manifold $S$ and two points $a,b\in S$, $\aUb^\bullet$ is a cosimplicial object with
    \begin{itemize}
        \item (components) $\aUb^n:=\underbrace{S\times \cdots\times S}_{n}$.
        \item (coface map) $\delta^i:\aUb^n\to \aUb^{n+1}$ ($i=0,\dots,n$) is given by $$\delta^i(s_1,\dots,s_n):=\begin{cases}
            (a,s_1,\dots,s_n), & i=0,\\
            (s_1,\dots,s_i,s_i,\dots,s_n), & 0<i<n,\\
            (s_1,\dots,s_n,b), & i=n.
        \end{cases}$$
        \item (codegeneracy map) $\sigma^i:\aUb^{n+1}\to \aUb^n$ ($i=0,\dots,n$) is given by $$\sigma^i(s_1,\dots,s_{n+1}):=(s_1,\dots,s_i,s_{i+2},\dots,s_{n+1}).$$
    \end{itemize}
\end{ex}

For given cosimplicial object $X^\bullet$ of an abelian category, we can define some cochain complexes of the abelian category.

\begin{dfn}
    For a cosimplicial object $X^\bullet$ of an abelian category with cofaces $\delta^i:I_n\to I_{n+1}$ and codegeneracies $\sigma^i:I_{n+1}\to I_n$, we define following two cochain complexes.
    \begin{enumerate}
        \item The associated cochain complex $CX^\bullet$ is defined by
            \begin{align}
                CX^n:=X^n, && \partial^n:\sum_{i=0}^{n+1}(-1)^i\delta^i:CX^n\to CX^{n+1}.
            \end{align}
        \item The associated normalized cochain complex $\mathcal{N}X^\bullet \subseteq CX^\bullet$ is defined by
            \begin{align}
                \mathcal{N}X^n:=\bigcap_{i=0}^{n-1}\Ker\sigma^i. 
            \end{align}
    \end{enumerate}
\end{dfn}

To prove the main theorem, we now introduce the other cochain complex $C^\bullet(I_N, X^\bullet)$ for each $N\in\bZge{0}$ and a cosimplicial object $X^\bullet$.

\begin{dfn}
    For each $\emptyset \ne I\subseteq I_N$, we denote $X^I:=X^{\# I}$. For $K=\{i_0,\dots,i_l\}\subseteq I_N$ with $i_0<\cdots<i_l$ and $I=K\setminus\{i_j\}$ ($j=0,\dots,l$), we denote $\epsilon(I,K):=(-1)^j$. We also write
    \begin{align}
        d_{K,I}:=\delta^i:X^I\to X^K.
    \end{align}
    Then, we define the truncated normalized cochain complex $C^\bullet(I_N, X^\bullet)$ is defined by
    \begin{align}
        C^n(I_N,X^\bullet):=\bigoplus_{I\in\binom{I_N}{n+1}}X^I
    \end{align}
    with differential map $d: C^n(I_N,X^\bullet)\to C^{n+1}(I_N,X^\bullet)$ defined by
    \begin{align}
        d:=\bigoplus_{\substack{I\in\binom{I_N}{n}, K\in\binom{I_N}{n+1}\\I\subseteq K}} \epsilon(I,K)d_{K,I}.
    \end{align}
\end{dfn}

We denote by $\sigma_{\le N}$ the b\^{e}te filtration
\begin{align}
    \sigma_{\le N}C^n:=\begin{cases}
        C^n, & n\ge N,\\
        0, & n>N.
    \end{cases}
\end{align}
Then, we have the following lemma.

\begin{lem}[{\cite[Proposition A.238, p.483]{GiFr}}] \label{lem:cosimplicial_vs_normalized}
    Let $X^\ast$ be a cosimplicial object in abelian category and $N\in\bZge{0}$, then $C^\ast(I_N,X^\ast)$ and $\sigma_{\le N} \mathcal{N}X^\ast$ are functorially homotopically equivalent.
\end{lem}

\bibliographystyle{plain}
\bibliography{bibs}

\end{document}